\documentclass[12pt,reqno]{amsart}

\usepackage{amsmath,amsfonts,amsthm,amssymb,amscd, dsfont,tipa,mathtools,listings, hyperref, tkz-graph, tikz, float}
\usetikzlibrary{graphs} 
\usetikzlibrary{positioning}

\newtheorem{thm}{Theorem}[section]

\newtheorem{cor}[thm]{Corollary}
\newtheorem{lem}[thm]{Lemma}

\newtheorem{rek}[thm]{Remark}

\title{Matchings with prescribed color counts}
\author{Tudor Popescu} \address[Tudor Popescu]{Brandeis University, Waltham, MA 02453}\email{tudorpopescu@brandeis.edu}
\date{}
\keywords{graph theory, perfect matchings, colored graphs, bipartite graphs}

\begin{document}

\maketitle
\begin{abstract}
    In this note, we prove an interesting result about perfect matchings in a \begin{math}K_{2n, 2n}\end{math} whose edges are colored in red and blue such that each vertex is part of \begin{math}n\end{math} red edges and \begin{math}n\end{math} blue edges.
\end{abstract} 

\section{Introduction} 
\noindent A classical problem in complexity theory is the Exact Matching (EM) Problem, where given an edge-colored graph, with colors red and blue, and an integer \begin{math}k,\end{math} the goal is to decide whether or not the graph contains a perfect matching with exactly \begin{math}k\end{math} red edges. The problem was first introduced in 1982 by Papadimitriou and Yannakakis in \emph{[2]}, and in 1987 Mulmuley, Vazirani, and Vazirani \emph{[3]} showed that there is a randomized algorithm for this problem. We consider the complete bipartite (\begin{math}K_{2n, 2n}\end{math}) variant of this problem where each vertex is part of \begin{math}n\end{math} blue edges and \begin{math}n\end{math} red ones. In this paper, we will prove that if \begin{math}k\end{math} is even, there always exists a perfect matching with \begin{math}k\end{math} red edges and \begin{math}2n - k\end{math} blue ones, and that the same holds when \begin{math}k\end{math} is odd unless the graph has a very specific structure (if the graph formed by blue edges is disconnected).

\section{Main Theorem}

\begin{thm}
\label{thm:main}
    Let \begin{math}k, n \in \mathbb{N}\end{math} with \begin{math}k \le 2n.\end{math} Color all the edges of a \begin{math}K_{2n, 2n}\end{math} in red and blue such that each vertex is incident with \begin{math}n\end{math} red edges and \begin{math}n\end{math} blue ones. Then there exists a perfect matching with \begin{math}k\end{math} red edges and \begin{math}2n - k\end{math} blue ones unless \begin{math}k\end{math} is odd and the graph of blue edges is isomorphic to \begin{math}K_{n, n} \cup K_{n, n}\end{math}.
\end{thm}

\noindent Before proving this, we show that for any \begin{math}k\end{math} odd, there exists a graph \begin{math}K_{2n, 2n}\end{math} that does not have a perfect matching with \begin{math}k\end{math} red edges and \begin{math}2n - k\end{math} blue ones.

\begin{lem}
\label{lem: odd}
    Let \begin{math}G_{\text{blue}}\end{math} be the graph of blue edges. If \begin{math}k\end{math} is odd and \begin{math}G_{\text{blue}}\end{math} is isomorphic to \begin{math}K_n \cup K_n,\end{math} then \begin{math}\not \exists\end{math} a perfect match with \begin{math}k\end{math} red edges and \begin{math}2n - k\end{math} blue ones.
\end{lem}

\begin{figure}[h!]
\centering
\begin{tikzpicture}
    \coordinate (X1) at (0, 4);
    \coordinate (X2) at (0, 3);
    \coordinate (X3) at (0, 1);
    \coordinate (X4) at (0, 0);
    
    \coordinate (Y1) at (3, 4);
    \coordinate (Y2) at (3, 3);
    \coordinate (Y3) at (3, 1);
    \coordinate (Y4) at (3, 0);
    
    \filldraw[black] (X1) circle (2pt) node[left] {\begin{math}a_1\end{math}};
    \filldraw[black] (X2) circle (2pt) node[left] {\begin{math}a_2\end{math}};
    \filldraw[black] (X3) circle (2pt) node[left] {\begin{math}a_3\end{math}};
    \filldraw[black] (X4) circle (2pt) node[left] {\begin{math}a_4\end{math}};
    
    \filldraw[black] (Y1) circle (2pt) node[right] {\begin{math}b_1\end{math}};
    \filldraw[black] (Y2) circle (2pt) node[right] {\begin{math}b_2\end{math}};
    \filldraw[black] (Y3) circle (2pt) node[right] {\begin{math}b_3\end{math}};
    \filldraw[black] (Y4) circle (2pt) node[right] {\begin{math}b_4\end{math}};
    
    \draw[red, thick] (X1) -- (Y1);
    \draw[red, thick] (X2) -- (Y2);
    \draw[red, thick] (X3) -- (Y3);
    \draw[red, thick] (X4) -- (Y4);
    \draw[red, thick] (X3) -- (Y4);
    \draw[red, thick] (X4) -- (Y3);
    \draw[red, thick] (X1) -- (Y2);
    \draw[red, thick] (X2) -- (Y1);
    
    \draw[blue, thick] (X1) -- (Y3);
    \draw[blue, thick] (X1) -- (Y4);
    \draw[blue, thick] (X2) -- (Y3);
    \draw[blue, thick] (X2) -- (Y4);
    \draw[blue, thick] (X3) -- (Y1);
    \draw[blue, thick] (X3) -- (Y2);
    \draw[blue, thick] (X4) -- (Y1);
    \draw[blue, thick] (X4) -- (Y2);
    
    \draw[dashed] (-0.5, 3.5) ellipse (0.7 and 1.3) node[left=0.7cm] {\begin{math}A_1\end{math}};
    \draw[dashed] (-0.5, 0.5) ellipse (0.7 and 1.3) node[left=0.7cm] {\begin{math}A_2\end{math}};
    
    \draw[dashed] (3.5, 3.5) ellipse (0.7 and 1.3) node[right=0.7cm] {\begin{math}B_1\end{math}};
    \draw[dashed] (3.5, 0.5) ellipse (0.7 and 1.3) node[right=0.7cm] {\begin{math}B_2\end{math}};

\end{tikzpicture}
\caption{The case when \begin{math}n = 2\end{math} and \begin{math}G_{\text{blue}} = K_2 \cup K_2\end{math}}
\label{fig:2k22_grouped}
\end{figure}
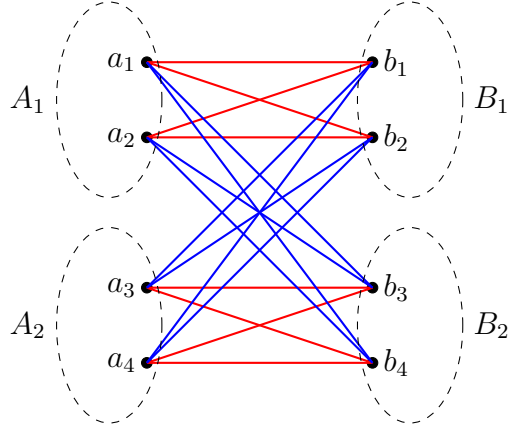

\begin{proof}
By symmetry between red and blue, we can assume that \begin{math}k \le n.\end{math} Indeed, if \begin{math}k > n,\end{math} just replace it with \begin{math}2n - k.\end{math}
\newline
Consider the complete bipartite graph with vertices \begin{math}\{a_1, \ldots, a_{2n}, b_1, \ldots, b_{2n}\}.\end{math} Let \begin{math}A_1 = \{a_1, \ldots, a_{n}\}, A_2 = \{a_{n + 1}, \ldots, a_{2n}\}, B_1 = \{b_1, \ldots, b_{n}\}, B_2 = \{b_{n + 1}, \ldots, b_{2n}\}.\end{math} Color the edges between \begin{math}A_1\end{math} and \begin{math}B_1\end{math} and the edges between \begin{math}A_2\end{math} and \begin{math}B_2\end{math} in red. Similarly, color the edges between \begin{math}A_1\end{math} and \begin{math}B_2\end{math} and the edges between \begin{math}A_2\end{math} and \begin{math}B_1\end{math} in blue. Therefore, every vertex will be a part of \begin{math}n\end{math} red edges and \begin{math}n\end{math} blue edges.
\newline
Assume for the sake of contradiction that this graph has a perfect matching \begin{math}M\end{math} with \begin{math}k\end{math} red edges and \begin{math}2n - k\end{math} blue edges. Let \begin{math}x\end{math} be the number of red edges in the matching \begin{math}M\end{math} between \begin{math}A_1\end{math} and \begin{math}B_1\end{math} and \begin{math}y\end{math} be the number of blue edges in \begin{math}M\end{math} between \begin{math}A_1\end{math} and \begin{math}B_2.\end{math} In particular, this forces the number of red edges in \begin{math}M\end{math} between \begin{math}A_2\end{math} and \begin{math}B_2\end{math} to be \begin{math}k - x\end{math} and the number of blue edges in \begin{math}M\end{math} between \begin{math}A_2\end{math} and \begin{math}B_1\end{math} to be \begin{math}2n - k - y.\end{math}
\newline
Since there are \begin{math}n\end{math} vertices in \begin{math}A_1\end{math} and we have a perfect matching given by the red and blue edges, we must have that 
\begin{displaymath}x + y = n\end{displaymath}
Similarly, there are \begin{math}n\end{math} vertices in \begin{math}B_1,\end{math} hence
\begin{displaymath}
    x + 2n - k - y = n
\end{displaymath}
Summing these up, we get that 
\begin{displaymath}
    2x + 2n - k = 2n \Rightarrow k = 2x
\end{displaymath}
However, this is a contradiction since \begin{math}k\end{math} is odd. Therefore, there does not exist a perfect matching with \begin{math}k\end{math} red edges and \begin{math}2n - k\end{math} blue edges, as desired.
\end{proof}


\begin{lem}
\label{lem: 02n}    
Color all the edges of a \begin{math}K_{2n, 2n}\end{math} with red and blue such that each vertex is part of \begin{math}n\end{math} red edges and \begin{math}n\end{math} blue ones. Then there exists a perfect matching with only blue edges.
\end{lem}

\begin{proof}
    Let \begin{math}G = (X, Y, E),\end{math} where \begin{math}X\end{math} is the set of vertices on the left, \begin{math}Y\end{math} is the set of vertices on the right, and \begin{math}E\end{math} is the set of blue edges. Note that \begin{math}\deg(v) = n, \forall v \in X \cup Y.\end{math} We need to show that \begin{math}G\end{math} has a perfect matching, and we'll do so using Hall's Marriage Lemma. Indeed, using this, we only have to show that \begin{math}\forall S \subset X,\end{math} the number of neighbors \begin{math}N(S)\end{math} of \begin{math}S,\end{math} is at least \begin{math}|S|.\end{math} We prove this by considering the following two cases:
    \begin{itemize}
        \item \begin{math}|S| \le n:\end{math} take \begin{math}x \in S.\end{math} Then \begin{math}N(x) \subseteq N(S),\end{math} so \begin{displaymath}|S| \le n = |N(x)| \le |N(S)|\end{displaymath}
        \item \begin{math}|S| > n:\end{math} we will actually show that \begin{math}N(S) = Y.\end{math} Indeed, assume for the sake of contradiction that \begin{math}\exists y \in Y \backslash N(S).\end{math} Therefore, \begin{math}N(y) \cap S = \emptyset.\end{math} But \begin{math}|S| + |N(y)| > n + n = |X|,\end{math} so by the pigeonhole principle we must have that \begin{math}N(y) \cap S \neq \emptyset,\end{math} contradicting our assumption. Therefore, we have that \begin{math}N(S) = Y,\end{math} so 
        \begin{displaymath}
            |S| \le 2n = |Y| = |N(s)
        \end{displaymath}
    \end{itemize}
    Therefore, using Hall's Marriage Lemma we obtain that \begin{math}G\end{math} contains a perfect matching, as desired.
\end{proof}

\noindent We are now ready to prove Theorem \ref{thm:main}.

\begin{proof}[Proof of Theorem \ref{thm:main}]

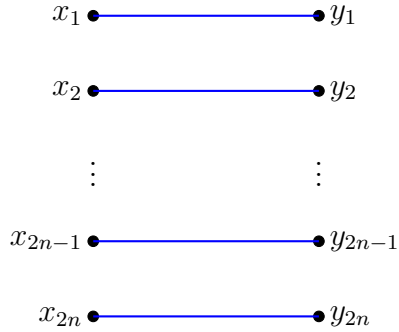
\begin{figure}[h!]
\centering
\begin{tikzpicture}
    \coordinate (X1) at (0, 4);
    \coordinate (X2) at (0, 3);
    \coordinate (XDots) at (0, 2); 
    \coordinate (X2n1) at (0, 1);
    \coordinate (X2n) at (0, 0);
    
    \coordinate (Y1) at (3, 4);
    \coordinate (Y2) at (3, 3);
    \coordinate (YDots) at (3, 2); 
    \coordinate (Y2n1) at (3, 1);
    \coordinate (Y2n) at (3, 0);
    
    \filldraw[black] (X1) circle (2pt) node[left] {\begin{math}x_1\end{math}};
    \filldraw[black] (X2) circle (2pt) node[left] {\begin{math}x_2\end{math}};
    \node at (XDots) {\vdots}; 
    \filldraw[black] (X2n1) circle (2pt) node[left] {\begin{math}x_{2n-1}\end{math}};
    \filldraw[black] (X2n) circle (2pt) node[left] {\begin{math}x_{2n}\end{math}};
    
    \filldraw[black] (Y1) circle (2pt) node[right] {\begin{math}y_1\end{math}};
    \filldraw[black] (Y2) circle (2pt) node[right] {\begin{math}y_2\end{math}};
    \node at (YDots) {\vdots}; 
    \filldraw[black] (Y2n1) circle (2pt) node[right] {\begin{math}y_{2n-1}\end{math}};
    \filldraw[black] (Y2n) circle (2pt) node[right] {\begin{math}y_{2n}\end{math}};
    
    \draw[blue, thick] (X1) -- (Y1);
    \draw[blue, thick] (X2) -- (Y2);
    \draw[blue, thick] (X2n1) -- (Y2n1);
    \draw[blue, thick] (X2n) -- (Y2n);
\end{tikzpicture}
\caption{Blue Perfect Match}
\label{fig:blue}
\end{figure}

We can assume without loss of generality that \begin{math}k \le n.\end{math} By Lemma \ref{lem: 02n}, there exists a perfect matching with \begin{math}2n\end{math} blue edges. By relabeling the vertices, we can assume that the blue edges in the perfect matching are \begin{math}(x_i, y_i),\end{math} as can be seen in figure \ref{fig:blue}. Our goal is to replace some of the blue edges \begin{math}(x_i, y_i)\end{math} in the matching with red edges of the form \begin{math}(x_{2l - 1}, y_{2l})\end{math} or \begin{math}(x_{2l}, y_{2l - 1}), \text{ for some } l \in \{1, \ldots, n\}.\end{math} The possible colorings of the subgraph with vertices \begin{math}x_{2l - 1}, x_{2l}, y_{2l - 1}, y_{2l}\end{math} are shown in figure \ref{fig:k22}.

\begin{figure}[h!]
    \centering

\begin{tikzpicture}[scale=1.5]

\tikzstyle{vertex}=[circle, fill, inner sep=1.5pt]

\begin{scope}
    \node[vertex, label=left:{\begin{math}x_{2l-1}\end{math}}] (x1) at (0, 1) {};
    \node[vertex, label=left:{\begin{math}x_{2l}\end{math}}] (x2) at (0, 0) {};
    \node[vertex, label=right:{\begin{math}y_{2l-1}\end{math}}] (y1) at (1, 1) {};
    \node[vertex, label=right:{\begin{math}y_{2l}\end{math}}] (y2) at (1, 0) {};
    
    \draw[blue, thick] (x1) -- (y1);
    \draw[blue, thick] (x2) -- (y2);
    \draw[blue, thick] (x1) -- (y2);
    \draw[blue, thick] (x2) -- (y1);
\end{scope}

\begin{scope}[xshift=2.8cm]
    \node[vertex, label=left:{\begin{math}x_{2l-1}\end{math}}] (x1) at (0, 1) {};
    \node[vertex, label=left:{\begin{math}x_{2l}\end{math}}] (x2) at (0, 0) {};
    \node[vertex, label=right:{\begin{math}y_{2l-1}\end{math}}] (y1) at (1, 1) {};
    \node[vertex, label=right:{\begin{math}y_{2l}\end{math}}] (y2) at (1, 0) {};
    
    \draw[blue, thick] (x1) -- (y1);
    \draw[blue, thick] (x2) -- (y2);
    \draw[blue, thick] (x1) -- (y2);
    \draw[red, thick] (x2) -- (y1);
\end{scope}

\begin{scope}[xshift=5.6cm]
    \node[vertex, label=left:{\begin{math}x_{2l-1}\end{math}}] (x1) at (0, 1) {};
    \node[vertex, label=left:{\begin{math}x_{2l}\end{math}}] (x2) at (0, 0) {};
    \node[vertex, label=right:{\begin{math}y_{2l-1}\end{math}}] (y1) at (1, 1) {};
    \node[vertex, label=right:{\begin{math}y_{2l}\end{math}}] (y2) at (1, 0) {};
    
    \draw[blue, thick] (x1) -- (y1);
    \draw[blue, thick] (x2) -- (y2);
    \draw[blue, thick] (x2) -- (y1);
    \draw[red, thick] (x1) -- (y2);
\end{scope}

\begin{scope}[xshift=8.4cm]
    \node[vertex, label=left:{\begin{math}x_{2l-1}\end{math}}] (x1) at (0, 1) {};
    \node[vertex, label=left:{\begin{math}x_{2l}\end{math}}] (x2) at (0, 0) {};
    \node[vertex, label=right:{\begin{math}y_{2l-1}\end{math}}] (y1) at (1, 1) {};
    \node[vertex, label=right:{\begin{math}y_{2l}\end{math}}] (y2) at (1, 0) {};
    
    \draw[blue, thick] (x1) -- (y1);
    \draw[blue, thick] (x2) -- (y2);
    \draw[red, thick] (x1) -- (y2);
    \draw[red, thick] (x2) -- (y1);
\end{scope}

\end{tikzpicture}

    \caption{Possible colorings of the graph with vertices \begin{math}\{x_{2l - 1},x_{2l}, y_{2l - 1}, y_{2l}\end{math}}
    \label{fig:k22}
\end{figure}
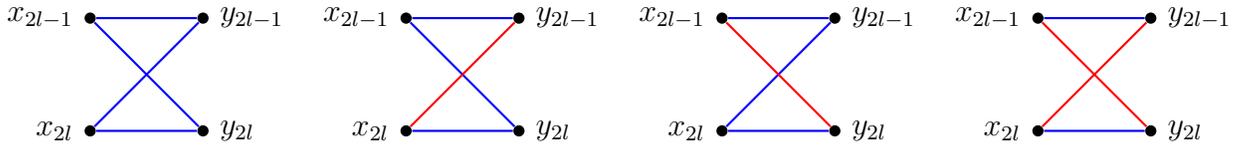

\noindent We wish to replace some of the blue edges \begin{math}(x_{2l - 1}, y_{2l - 1}), (x_{2l}, y_{2l})\end{math} in our perfect matching with edges \begin{math}(x_{2l - 1}, y_{2l}), (x_{2l}, y_{2l - 1}),\end{math} with at least one of the latter edges being red. Obviously, the only configuration where this does not work out is the leftmost one in figure \ref{fig:k22}. Let us assume that \begin{math}l\end{math} is an integer for which \begin{math}(x_{2l - 1}, y_{2l}), (x_{2l}, y_{2l - 1})\end{math} are blue. In this case there are \begin{math}n\end{math} red edges in the set \begin{math}\{(x_{2l - 1}, y_i\}), i \neq 2l - 1, 2l\}\end{math} and \begin{math}n\end{math} red edges in the set \begin{math}\{(x_{i}, y_{2l}\}), i \neq 2l - 1, 2l\}.\end{math} We will now use the pigeonhole principle; we create \begin{math}2n\end{math} "boxes" as follows: for \begin{math}i \in \{1, \ldots, n\}\end{math} if \begin{math}(x_{2l - 1}, y_{2i})\end{math} is red, we put it in box \begin{math}B_{2i - 1};\end{math} if \begin{math}(x_{2i - 1}, y_{2l})\end{math} is red, we put it in box \begin{math}B_{2i};\end{math} if \begin{math}(x_{2l - 1}, y_{2i - 1})\end{math} is red, we put it in box \begin{math}B_{2i};\end{math} if \begin{math}(x_{2l}, y_{2i})\end{math} is red, we put it in box \begin{math}B_{2i};\end{math} note that boxes \begin{math}B_{2l - 1}\end{math} and \begin{math}B_{2l}\end{math} are empty by construction. Therefore, we have at most \begin{math}2n - 2\end{math} nonempty "boxes", but \begin{math}2n\end{math} "pigeons" (the red edges of the above types). Therefore, by the pigeonhole principle, there exists \begin{math}i\end{math} such that either both \begin{math}(x_{2l - 1}, y_{2i}) \text{ and } (x_{2i - 1}, y_{2l})\end{math} are red, or both \begin{math}(x_{2l - 1}, y_{2i - 1}) \text{ and } (x_{2l}, y_{2i})\end{math} are red. If both \begin{math}(x_{2l - 1}, y_{2i}) \text{ and } (x_{2i - 1}, y_{2l})\end{math} are red, then relabeling \begin{math}x_{2l} \leftrightarrow x_{2i}\end{math} and \begin{math}y_{2l} \leftrightarrow y_{2i}\end{math} we get a new configuration of the desired form (namely the edges \begin{math}(x_{2l- 1}, y_{2l})\end{math} and \begin{math}(x_{2i - 1}, y_{2i})\end{math} will be red; this process can be visualized in figure \ref{fig:interchanging}. The proof when both \begin{math}(x_{2l - 1}, y_{2i - 1}) \text{ and } (x_{2l}, y_{2i})\end{math} are red is analogous. We repeat the process until there is no \begin{math}i\end{math} with both \begin{math}(x_{2i - 1}, y_{2i}), (x_{2i}, y_{2i - 1})\end{math} blue. 

\noindent We now distinguish two cases depending on the parity of \begin{math}k.\end{math} Suppose first that \begin{math}k\end{math} is even. If there are at least \begin{math}k/2\end{math} integers \begin{math}i\end{math} with both \begin{math}(x_{2i - 1}, y_{2i}), (x_{2i}, y_{2i - 1})\end{math} red, then we can choose exactly \begin{math}k\end{math} red edges \begin{math}(x_{2i - 1}, y_{2i})\end{math} and \begin{math}(x_{2i}, y_{2i - 1})\end{math} from these (since \begin{math}k\end{math} is even) and take the rest \begin{math}2n - k\end{math} blue edges of the form \begin{math}(x_j, y_j)\end{math} to create a perfect matching. If there are less \begin{math}k/2\end{math} integers \begin{math}i\end{math} with both \begin{math}(x_{2i - 1}, y_{2i}), (x_{2i}, y_{2i - 1})\end{math} red, choose all such edges. Then pick enough red edges of the form \begin{math}(x_{2j - 1}, y_{2j})\end{math} or \begin{math}(x_{2j}, y_{2j - 1})\end{math} so that we have chosen exactly \begin{math}k\end{math} red edges. We can do this since \begin{math}k \le n\end{math} and by our algorithm, at least one of edges \begin{math}(x_{2i - 1}, y_{2i})\end{math} and \begin{math}(x_{2i}, y_{2i - 1})\end{math} is red. For the blue edges, choose \begin{math}(x_{2i - 1}, y_{2i})\end{math} if it is blue and \begin{math}(x_{2i}, y_{2i - 1})\end{math} was one of the chosen red edges, \begin{math}(x_{2i}, y_{2i - 1})\end{math} if it is blue and \begin{math}(x_{2i - 1}, y_{2i})\end{math} was one of the chosen red edges, and \begin{math}(x_i, y_i),\end{math} if \begin{math}x_i, y_i\end{math} are not part of any chosen red edge. In this way, we obtain a perfect matching with \begin{math}k\end{math} red edges and \begin{math}2n - k\end{math} blue edges, as desired.
 
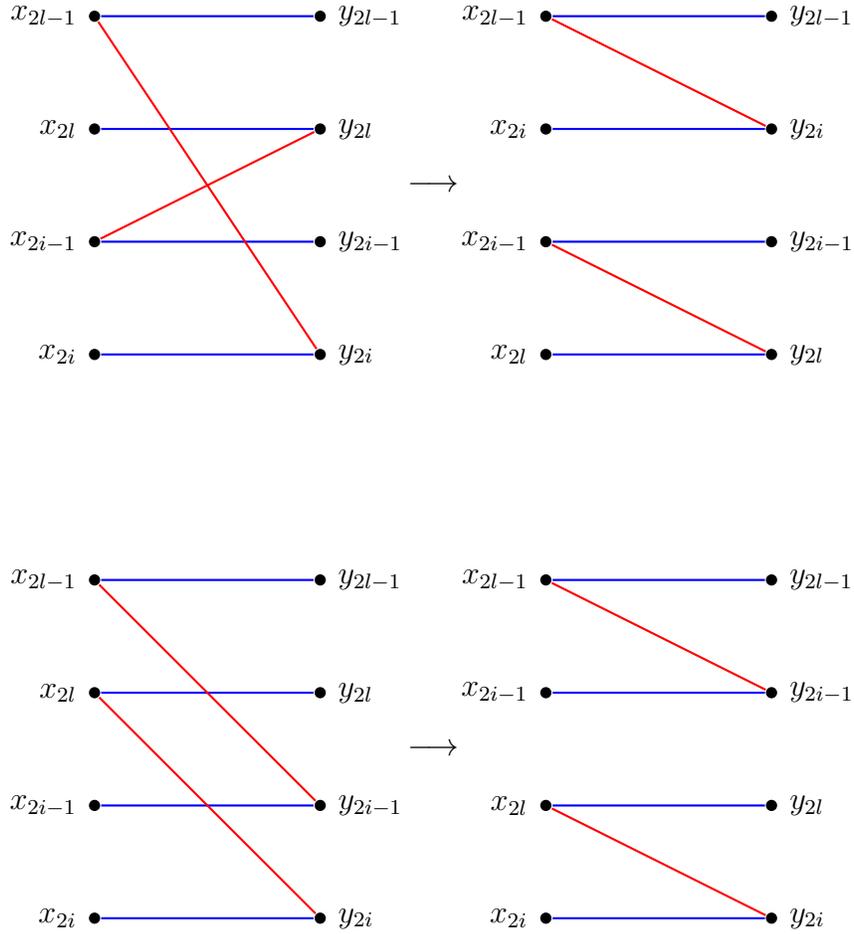
\begin{figure}[h!]
    \centering
    \begin{tikzpicture}[scale=1.5, every node/.style={draw=none}, every path/.style={thick}]

\node[circle, fill=black, inner sep=1.5pt, label=left:\begin{math}x_{2l-1}\end{math}] (x2k-1) at (0, 3) {};
\node[circle, fill=black, inner sep=1.5pt, label=left:\begin{math}x_{2l}\end{math}] (x2k) at (0, 2) {};
\node[circle, fill=black, inner sep=1.5pt, label=left:\begin{math}x_{2i-1}\end{math}] (x2i-1) at (0, 1) {};
\node[circle, fill=black, inner sep=1.5pt, label=left:\begin{math}x_{2i}\end{math}] (x2i) at (0, 0) {};

\node[circle, fill=black, inner sep=1.5pt, label=right:\begin{math}y_{2l-1}\end{math}] (y2k-1) at (2, 3) {};
\node[circle, fill=black, inner sep=1.5pt, label=right:\begin{math}y_{2l}\end{math}] (y2k) at (2, 2) {};
\node[circle, fill=black, inner sep=1.5pt, label=right:\begin{math}y_{2i-1}\end{math}] (y2i-1) at (2, 1) {};
\node[circle, fill=black, inner sep=1.5pt, label=right:\begin{math}y_{2i}\end{math}] (y2i) at (2, 0) {};

\draw[blue] (x2k-1) -- (y2k-1);
\draw[blue] (x2k) -- (y2k);
\draw[blue] (x2i-1) -- (y2i-1);
\draw[blue] (x2i) -- (y2i);

\draw[red] (x2k-1) -- (y2i);
\draw[red] (x2i-1) -- (y2k);

\node[draw=none, rectangle] at (3, 1.5) {\begin{math}\longrightarrow\end{math}};

\node[circle, fill=black, inner sep=1.5pt, label=left:\begin{math}x_{2l-1}\end{math}] (x2k-1_r) at (4, 3) {};
\node[circle, fill=black, inner sep=1.5pt, label=left:\begin{math}x_{2i}\end{math}] (x2i_r) at (4, 2) {};
\node[circle, fill=black, inner sep=1.5pt, label=left:\begin{math}x_{2i-1}\end{math}] (x2i-1_r) at (4, 1) {};
\node[circle, fill=black, inner sep=1.5pt, label=left:\begin{math}x_{2l}\end{math}] (x2k_r) at (4, 0) {};

\node[circle, fill=black, inner sep=1.5pt, label=right:\begin{math}y_{2l-1}\end{math}] (y2k-1_r) at (6, 3) {};
\node[circle, fill=black, inner sep=1.5pt, label=right:\begin{math}y_{2i}\end{math}] (y2i_r) at (6, 2) {};
\node[circle, fill=black, inner sep=1.5pt, label=right:\begin{math}y_{2i-1}\end{math}] (y2i-1_r) at (6, 1) {};
\node[circle, fill=black, inner sep=1.5pt, label=right:\begin{math}y_{2l}\end{math}] (y2k_r) at (6, 0) {};

\draw[blue] (x2k-1_r) -- (y2k-1_r);
\draw[blue] (x2k_r) -- (y2k_r);
\draw[blue] (x2i-1_r) -- (y2i-1_r);
\draw[blue] (x2i_r) -- (y2i_r);

\draw[red] (x2k-1_r) -- (y2i_r);
\draw[red] (x2i-1_r) -- (y2k_r);


\node[circle, fill=black, inner sep=1.5pt, label=left:\begin{math}x_{2l-1}\end{math}] (x2k-1_b) at (0, -2) {};
\node[circle, fill=black, inner sep=1.5pt, label=left:\begin{math}x_{2l}\end{math}] (x2k_b) at (0, -3) {};
\node[circle, fill=black, inner sep=1.5pt, label=left:\begin{math}x_{2i-1}\end{math}] (x2i-1_b) at (0, -4) {};
\node[circle, fill=black, inner sep=1.5pt, label=left:\begin{math}x_{2i}\end{math}] (x2i_b) at (0, -5) {};

\node[circle, fill=black, inner sep=1.5pt, label=right:\begin{math}y_{2l-1}\end{math}] (y2k-1_b) at (2, -2) {};
\node[circle, fill=black, inner sep=1.5pt, label=right:\begin{math}y_{2l}\end{math}] (y2k_b) at (2, -3) {};
\node[circle, fill=black, inner sep=1.5pt, label=right:\begin{math}y_{2i-1}\end{math}] (y2i-1_b) at (2, -4) {};
\node[circle, fill=black, inner sep=1.5pt, label=right:\begin{math}y_{2i}\end{math}] (y2i_b) at (2, -5) {};

\draw[blue] (x2k-1_b) -- (y2k-1_b);
\draw[blue] (x2k_b) -- (y2k_b);
\draw[blue] (x2i-1_b) -- (y2i-1_b);
\draw[blue] (x2i_b) -- (y2i_b);

\draw[red] (x2k-1_b) -- (y2i-1_b);
\draw[red] (x2k_b) -- (y2i_b);

\node[draw=none, rectangle] at (3, -3.5) {\begin{math}\longrightarrow\end{math}};

\node[circle, fill=black, inner sep=1.5pt, label=left:\begin{math}x_{2l-1}\end{math}] (x2k-1_r_b) at (4, -2) {};
\node[circle, fill=black, inner sep=1.5pt, label=left:\begin{math}x_{2i-1}\end{math}] (x2i-1_r_b) at (4, -3) {};
\node[circle, fill=black, inner sep=1.5pt, label=left:\begin{math}x_{2l}\end{math}] (x2k_r_b) at (4, -4) {};
\node[circle, fill=black, inner sep=1.5pt, label=left:\begin{math}x_{2i}\end{math}] (x2i_r_b) at (4, -5) {};

\node[circle, fill=black, inner sep=1.5pt, label=right:\begin{math}y_{2l-1}\end{math}] (y2k-1_r_b) at (6, -2) {};
\node[circle, fill=black, inner sep=1.5pt, label=right:\begin{math}y_{2i-1}\end{math}] (y2i-1_r_b) at (6, -3) {};
\node[circle, fill=black, inner sep=1.5pt, label=right:\begin{math}y_{2l}\end{math}] (y2k_r_b) at (6, -4) {};
\node[circle, fill=black, inner sep=1.5pt, label=right:\begin{math}y_{2i}\end{math}] (y2i_r_b) at (6, -5) {};

\draw[blue] (x2k-1_r_b) -- (y2k-1_r_b);
\draw[blue] (x2k_r_b) -- (y2k_r_b);
\draw[blue] (x2i-1_r_b) -- (y2i-1_r_b);
\draw[blue] (x2i_r_b) -- (y2i_r_b);

\draw[red] (x2k-1_r_b) -- (y2i-1_r_b);
\draw[red] (x2k_r_b) -- (y2i_r_b);

    \end{tikzpicture}
    \caption{Interchanging the vertices yields the desired form}
    \label{fig:interchanging}
\end{figure}

\noindent Assume now that \begin{math}k\end{math} be odd. Arguing the same as above, we see that if there exist \begin{math}i, j, i \neq j\end{math} with \begin{math}(x_i, y_i), (x_j, y_j), (x_i, y_j)\end{math} blue and \begin{math}(x_j, y_i)\end{math} red, we have a perfect matching with \begin{math}k\end{math} red edges and \begin{math}2n - k\end{math} blue edges. Therefore, the only case in which the proof can not be recreated is when (even after interchanging some vertices) for any \begin{math}i, j,\end{math} we have that \begin{math}(x_i, y_i) \text{ and } (x_j, y_j)\end{math} are blue, and \begin{math}(x_i, y_j) \text{ and } (x_j, y_i)\end{math} are either both red or both blue.

\noindent We show that in this case, if \begin{math}(x_i, y_j), (x_j, y_i), (x_i, y_l), (x_l, y_i)\end{math} are blue, then we must have that \begin{math}(x_j, y_l), (x_l, y_j)\end{math} are blue as well. Assume for the sake of contradiction that they are both red (we know that they are both of the same color). Then interchanging \begin{math}y_i \leftrightarrow y_l,\end{math} we get that \begin{math}(x_i, y_i), (x_j, y_j), (x_l, y_l), (x_i, y_j), (x_i, y_l), (x_j, y_l), (x_l, y_i)\end{math} are all blue, but \begin{math}(x_j, y_i)\end{math} and \begin{math}(x_l, y_j)\end{math} are red. In this case, the subgraph with vertices \begin{math}x_i, x_j, y_i, y_j\end{math} has blue edges \begin{math}(x_i, y_i), (x_j, y_j),\end{math} and \begin{math}(x_i, y_j),\end{math} and red edge \begin{math}(x_j, y_i).\end{math} But we are in the case where no such configuration exists, so we get a contradiction. Therefore, we must have that the graph of blue edges is the union of two complete \begin{math}K_{n, n}\end{math} graphs.
\end{proof}

\section{Concluding Remakrs}

\begin{rek}
    Let \begin{math}G = (X, Y, E)\end{math} be a bipartite graph with \begin{math}|X| = |Y| = 2n\end{math} such that \begin{math}\deg(v) = n, \forall v \in X \cup Y.\end{math} Then \begin{math}G\end{math} is connected, unless \begin{math}G\end{math} is isomorphic to a disjoint union of two copies of \begin{math}K_{n, n}.\end{math}
\end{rek}

\begin{proof}

Assume that \begin{math}G\end{math} is disconnected. If there is no path between \begin{math}x_i\end{math} and \begin{math}x_j\end{math}, then \begin{math}N(x_i) \cap N(x_j) = \emptyset.\end{math} Without loss of generality we can assume that \begin{math}N(x_i) = \{y_ 1, \ldots, y_n\}, N(x_j) = \{y_{n + 1}, \ldots, y_{2n}\}.\end{math} Since there is no path between \begin{math}x_i\end{math} and \begin{math}x_j,\end{math} we must have that \begin{math}N(y_a) \cap N(y_b) = \emptyset, \forall a \in \{1, \ldots, n\}, b \in \{n + 1, 2n\}.\end{math} But since \begin{math}\deg(v) = n, \forall v,\end{math} we have that \begin{math}N(y_{a_1}) = N(y_{a_2}), \forall a_1, a_2 \in \{1, 2, \ldots, n\}\end{math} and \begin{math}N(y_{b_1}) = N(y_{b_2}), \forall b \in \{n + 1, 2n\}.\end{math} Therefore, the graph is isomorphic to two copies of \begin{math}K_{n, n}.\end{math} The proof when there is no path between \begin{math}x_i\end{math} and \begin{math}y_j\end{math} is analogous.

\end{proof}

\noindent Using the above remark, we easily obtain a reformulation of Theorem \ref{thm:main}:

\begin{cor}
    Let \begin{math}k, n \in \mathbb{N}\end{math} with \begin{math}k\end{math} odd. Color all the edges of a \begin{math}K_{2n, 2n}\end{math} with red and blue such that each vertex is part of \begin{math}n\end{math} red edges and \begin{math}n\end{math} blue ones. Then there exists a perfect matching with \begin{math}k\end{math} red edges and \begin{math}2n - k\end{math} blue ones if and only if the graph of blue edges is connected.
\end{cor}

\section{Application to the permanent}

\noindent There has been extensive research on the permanent of a matrix with \begin{math}0, 1\end{math} entries. In 1963, Minc conjectured that if \begin{math}A\end{math} is a square binary matrix of size \begin{math}n\end{math} and \begin{math}r_i = a_{i, 1} + \ldots + a_{i, n},\end{math} then \begin{math}0 \le \text{per}(A) \le \prod_{i = 1}^{n} (r_i!)^{1/r_i};\end{math} this was proved by Bregman \emph{[1]} in 1973. We can rephrase Theorem \ref{thm:main} as follows:

\begin{cor}
    Let \begin{math}A\end{math} be a square matrix of size \begin{math}2n\end{math} with \begin{math}0, 1\end{math} entries such that the sum of the entries in each row and column is \begin{math}n.\end{math} Moreover, let \begin{math}k\end{math} be an even integer, \begin{math}0 \le k \le n.\end{math} There is an algorithm that selects \begin{math}k\end{math} zero entries, removes the rows and columns on which these entries lie, and assures that the remaining square binary matrix of size \begin{math}2n - k\end{math} has nonzero permanent.
\end{cor}

\section{Acknowledgement}
\noindent I would like to thank Kiril Bangachev, Olivier Bernardi, Wesley Pegden, and Kristof Zolomy for our extremely helpful discussions regarding this problem.

\end{document}